\documentclass{article}

\usepackage{amsmath,amssymb,amsfonts,amsthm,graphicx, float, lineno, hyperref, authblk,xcolor}
\usepackage[utf8]{inputenc}
\usepackage{hyperref}
\usepackage{mathtools}

\newtheorem{theorem}{Theorem}[section]
\newtheorem{lemma}[theorem]{Lemma}
\newtheorem{proposition}[theorem]{Proposition}

\newtheorem{definition}[theorem]{Definition}
\newtheorem{remark}[theorem]{Remark}

\numberwithin{equation}{section}

\title{Hausdorff Dimension of Growth Rate Level Sets in $\theta$-expansions}

\author[1]{Andreas Rusu}
\affil[1]{Faculty of Applied Sciences, National University of Science and Technology POLITEHNICA Bucharest, Splaiul Independentei 313, 060042 Bucharest, Romania\\ 
e-mail: \href{mailto: andreasrusu10@gmail.com}{andreasrusu10@gmail.com}}

\author[2,3]{Gabriela Ileana Sebe}
\affil[2]{Faculty of Applied Sciences, National University of Science and Technology POLITEHNICA Bucharest, Splaiul Independentei 313, 060042 Bucharest, Romania }
\affil[3]{Gheorghe Mihoc-Caius Iacob Institute of Mathematical Statistics and Applied Mathematics of the Romanian Academy, Calea 13 Sept. 13, 050711 Bucharest, Romania\\ e-mail: \href{mailto: igsebe@yahoo.com}{igsebe@yahoo.com}}

\date{September 2025}

\begin{document}

\maketitle

\begin{abstract}
We investigate the Hausdorff dimension of level sets defined by digit growth rates in $\theta$-expansions, a generalization of regular continued fractions. For any $\alpha \geq 0$, we prove that the set
\[
E_\theta(\alpha) = \left\{ x \in [0, \theta] \setminus \mathbb{Q} : \lim_{n \to {+}\infty} \frac{L_{n,\theta}(x) \log n \log \log n}{S_{n,\theta}(x) - L_{n,\theta}(x)} = \alpha \right\}
\]
has full Hausdorff dimension. This extends previous work of Zhang and {L\"u} (2016) on regular continued fractions to the broader framework of $\theta$-expansions. The proof involves constructing explicit subsets with controlled digit growth and establishing dimension preservation through H\"older-continuous mappings.
\end{abstract}

\section{Introduction}

The study of exceptional sets in the metric theory of continued fractions reveals a rich interplay between ergodic theory, probability, and fractal geometry. Classical limit theorems describe the almost-everywhere behavior of digit sequences: Khinchine established the weak law of large numbers for the sum of digits, while Lévy's theorem characterizes the growth of denominators. A fundamental insight, crystallized in the work of Diamond and Vaaler \cite{diamondvaaler}, is that while the arithmetic mean of digits diverges almost everywhere, removing the largest digit yields a stabilized sum that satisfies a strong law of large numbers. Philipp \cite{Philipp} further demonstrated that no regular norming sequence can yield a finite, non-zero almost sure limit for the full digit sum, confirming the inherent instability caused by extreme values. This highlights the delicate interplay between typical behavior and the influence of extreme digits in continued fraction expansions.

The Hausdorff dimension provides a finer tool to quantify sets exhibiting atypical behavior. Jarník's classical result showed that the set of badly approximable numbers, while null, has full Hausdorff dimension. Diamond and Vaaler \cite{diamondvaaler} extended this paradigm by proving that the set of real numbers for which the arithmetic mean of digits diverges has full Hausdorff dimension. 

This line of inquiry was significantly advanced by Zhang and L\"u \cite{ZhangLu}, who investigated the relative growth rate between the largest digit and the sum of the remaining digits. They proved that for any $\alpha \geq 0$, the level set
\[
\left\{ x \in [0,1] \setminus \mathbb{Q} : \lim_{n \to +\infty} \frac{L_n(x) \log n \log \log n}{S_n(x) - L_n(x)} = \alpha \right\}
\]
has full Hausdorff dimension, where $S_n(x)$ and $L_n(x)$ denote the sum and maximum of the first $n$ digits, respectively. 

In this paper, we extend the result beyond the classical regular continued fraction expansion to the broader framework of $\theta$-expansions. Let $\theta \in (0,1)$ be an irrational number such that $\theta^2 = 1/m$ for some $m \in \mathbb{N}_+$, where $m$ is not a perfect square. Every irrational $x \in [0, \theta]$ admits a unique $\theta$-expansion:
\begin{equation} \label{1.1}
{x = \cfrac{1}{\theta\ell_1(x) + \cfrac{1}{\theta\ell_2(x) + \cfrac{1}{\theta\ell_3(x) + \ddots}}}} =: [\ell_1(x), \ell_2(x), \ell_3(x), \ldots]_\theta,
\end{equation}
with digits satisfying $\ell_n(x) \geq m$ for all $n \geq 1$. Define the sum and maximum functions as
\[
S_{n,\theta}(x) = \sum_{k=1}^n \ell_k(x), \quad L_{n,\theta}(x) = \max_{1 \leq k \leq n} \ell_k(x).
\]


{Note that the regular continued fraction expansion corresponds to the special case $\theta = 1$. 
Although $\theta = 1$ is not included in our parameter set (since we require $\theta$ irrational and $\theta^2 = 1/m$ with $m$ not a perfect square), 
the family of $\theta$-expansions shares the essential dynamical and metric properties with the regular continued fraction. 
Thus our results extend those of Zhang and L\"u (2016) to this broader framework.}

{Our investigation forms a natural companion to the established metric theory of $\theta$-expansions. }

{
The present paper complements the measure-theoretic perspective of the limit theorems by studying the \emph{size} of these exceptional sets through the finer lens of fractal geometry. While the limit theorems describe what happens typically with respect to $\gamma_\theta$, the Hausdorff dimension results quantify how prevalent these exceptional behaviors are within the topological or dimensional sense.
}

Our main results demonstrate that despite the strong almost-everywhere constraints revealed by the limit theorems, the level sets of digit growth rates remain large in the sense of dimension.

\begin{theorem}\label{thm:main}
{Let $\theta\in(0,1)$ be an irrational such that $\theta^2=1/m$ for some non‑square $m\in\mathbb{N}_+$. }
For any $\alpha \geq 0$, the set
\[
E_\theta(\alpha) = \left\{ x \in [0, \theta] \setminus \mathbb{Q} : \lim_{n \to +\infty} \frac{L_{n,\theta}(x) \log n \log \log n}{S_{n,\theta}(x) - L_{n,\theta}(x)} = \alpha \right\}
\]
satisfies $\dim_H E_\theta(\alpha) = 1$.
\end{theorem}


Theorem \ref{thm:main} generalizes the work of Zhang and L\"u \cite{ZhangLu}.  Together with our generalizations of the Khinchine, Diamond-Vaaler, and Philipp theorems \cite{RusuSebeLascu}, they provide a comprehensive description of both the typical and exceptional digit growth in $\theta$-expansions.

The proofs require substantial adaptations to the $\theta$-expansion framework. A key innovation is the use of a novel sparse sequence $n_k = \lfloor \exp(k^{3/4}) \rfloor$, which differs from the polynomial sequences used in previous works and is better suited to the metric properties of $\theta$-expansions. Furthermore, we provide complete and detailed verifications of all metrical estimates, including a careful analysis of a Hölder-continuous map that preserves dimension between our constructed subsets and sets of numbers with bounded digits.

The paper is structured as follows. Section 2 contains the necessary preliminaries on $\theta$-expansions, including their metrical properties and a Jarn\'ik-type theorem. Section 3 presents the proof of Theorem \ref{thm:main}. We conclude with some final remarks in Section 4.

\section{Preliminaries on $\theta$-Expansions}

This section provides the necessary background on $\theta$-expansions, a generalization of regular continued fractions. We outline their dynamical construction, key metric properties, and the associated Jarník-type theorem for sets with bounded digits. For a comprehensive treatment, we refer to \cite{SebeLascuCraiova,Sebe2017,SebeLascu2019,SebeLascu-Springer,SebeLascuBilel2}.

\subsection{Dynamical System and Digit Sequences}

The $\theta$-expansion is generated by a specific iterated function system.
{Throughout, let $\theta\in(0,1)$ be an irrational such that $\theta^2=1/m$ for some non‑square $m\in\mathbb{N}_+$.}

\begin{definition}[Generalized Gauss Map]
The generalized Gauss map $T_\theta: [0, \theta] \to [0, \theta]$ is defined by
\[
T_\theta(x) = \begin{cases}
\frac{1}{x} - \theta \left\lfloor \frac{1}{\theta x} \right\rfloor & \text{for } x \in (0, \theta], \\
0 & \text{for } x = 0.
\end{cases}
\]
For $n \geq 1$, the $n$-th digit of $x$ is defined as
\[
\ell_n(x) = \left\lfloor \frac{1}{\theta T_\theta^{n-1}(x)} \right\rfloor,
\]
{with the convention that $\ell_n(0) = +\infty$ and $T_\theta^0(x) = x$}. 
\end{definition}

The convergents of $x$, which provide its best rational approximations in this framework, are given by $\frac{P_n(x)}{Q_n(x)} = [\ell_1(x), \ell_2(x), \ldots, \ell_n(x)]_{\theta}$. 
The sequences $\{P_n(x)\}$ and $\{Q_n(x)\}$ satisfy the recurrence relations:
\begin{align*}
P_n(x) &= \ell_n(x)\theta P_{n-1}(x) + P_{n-2}(x), \quad &P_{-1} = 1,\ P_0 = 0, \\
Q_n(x) &= \ell_n(x)\theta Q_{n-1}(x) + Q_{n-2}(x), \quad & Q_{-1} = 0,\ Q_0 = 1.
\end{align*}
The denominator $Q_n(x)$ plays a crucial role in metrical estimates.

\subsection{Ergodic and metric  properties}

The dynamical system $([0, \theta], T_\theta)$ admits an absolutely continuous invariant probability measure, generalizing the Gauss measure {(see, e.g., \cite{CR-2003})}.

\begin{definition}[$\theta$-Gauss Measure]
The $\theta$-Gauss measure $\gamma_\theta$ on $[0, \theta]$ is defined by its density
\[
\frac{\mathrm{d}\gamma_\theta}{\mathrm{d}x}(x) = h_\theta(x) = \frac{1}{\log(1+\theta^2)} \cdot \frac{\theta}{1 + \theta x}.
\]
The system $(T_\theta, \gamma_\theta)$ is ergodic.
\end{definition}

For an admissible sequence of digits $(\ell_1, \ldots, \ell_n)$ with $\ell_i \geq m$, the corresponding fundamental (or cylinder) interval is defined as
\[
I_n(\ell_1, \ldots, \ell_n) = \{x \in [0, \theta] \setminus \mathbb{Q} : \ell_1(x) = \ell_1, \ldots, \ell_n(x) = \ell_n\}.
\]
These intervals form the building blocks for our dimension calculations. The following proposition summarizes their key geometric properties, which are fundamental to the proofs in Sections 3 and 4.

\begin{proposition}[Metric Properties]\label{prop:basic}
For any $n \geq 1$ and admissible $\ell_1, \ldots, \ell_n \geq m$, let $Q_n = Q_n(\ell_1, \ldots, \ell_n)$. The following bounds hold:
\begin{enumerate}
    \item[(a)] (Exponential Growth) $Q_n \geq (m+1)^{\frac{n-1}{2}}$.
    \item[(b)] (Length Estimate) $\displaystyle \frac{\theta}{(1+\theta^2)Q_n^2} \leq |I_n(\ell_1, \ldots, \ell_n)| \leq \frac{\theta}{Q_n^2}$.
    \item[(c)] (Multiplicative Sensitivity) For any $1 \leq k \leq n$,
    \[
    \frac{(\ell_k+m)\theta}{2} \leq \frac{Q_n(\ell_1, \ldots, \ell_n)}{Q_{n-1}(\ell_1, \ldots, \ell_{k-1}, \ell_{k+1}, \ldots, \ell_n)} \leq (\ell_k+m)\theta.
    \]
    {Here $Q_{n-1}(\ell_1, \ldots, \ell_{k-1}, \ell_{k+1}, \ldots, \ell_n)$ denotes the denominator of the $(n-1)$-term convergent obtained by omitting the $k$-th digit.}
    This implies that the effect of a single digit $\ell_k$ on the total denominator $Q_n$ is roughly proportional to $\ell_k$.
\end{enumerate}
\end{proposition}
\noindent These properties are established in \cite{SebeLascuBilel2}.

Property (b) shows that the length of a cylinder is inversely proportional to the square of its denominator, a classic feature of continued fraction expansions. Property (c) is particularly important for our construction, as it allows us to control the distortion caused by inserting a single, very large digit into an otherwise bounded sequence.

\subsection{A Jarník-Type Theorem and Dimension Preservation}

A classical result by Jarník states that the set of badly approximable numbers (those with bounded regular continued fraction digits) has full Hausdorff dimension. The following theorem proved recently in \cite{SebeLascuBilel2} establishes an analogous result for $\theta$-expansions, providing the foundational ``full dimension" set used in our constructions.

\begin{theorem}[Jarník-type result for $\theta$-expansions]\label{thm:jarnik}
For any $M > 2m+1$, the set of numbers with digits uniformly bounded by $M$,
\[
X_M = \left\{ x \in [0, \theta] \setminus \mathbb{Q} : m \leq \ell_n(x) \leq M \text{ for all } n \geq 1 \right\},
\]
satisfies the dimension bounds
\[
1 - \frac{2(m+1)}{M+1} \cdot \frac{1}{\log(m+1)} \leq \dim_H (X_M) \leq 1 - \frac{m}{M+2} \cdot \frac{1}{\log \left( \frac{2M(M+1)}{m} \right)}.
\]
{Note that $\log(m+1)$ is a constant depending only on $\theta$, since $m$ is determined by $\theta$.}
Consequently, the set of all numbers with bounded digits,
\[
X = \left\{ x \in [0, \theta] \setminus \mathbb{Q} : \sup_{n \geq 1} \ell_n(x) < +\infty \right\},
\]
satisfies $\dim_H(X) = 1$.
\end{theorem}

The proof of this theorem is based on the metrical properties of Proposition \ref{prop:basic} and a standard mass distribution argument. The bounds show that $\dim_H (X_M) \to 1$ as $M \to +\infty$.

Finally, we recall a standard lemma from fractal geometry {(see, e.g., \cite{Falconer})} that will be instrumental in preserving Hausdorff dimension under the maps we construct. These maps will remove the sparsely inserted large digits, sending our exceptional sets into the bounded-digit sets $X_M$.

\begin{lemma}[Hölder Maps and Dimension]\label{lem:holder}
Let $F \subset \mathbb{R}$ and $f: F \to \mathbb{R}$ be a function such that
\[
|f(x)-f(y)| \leq c|x-y|^\gamma \quad \text{for all } x, y \in F,
\]
with constants $c > 0$ and $0 < \gamma \leq 1$. Then $\dim_H f(F) \leq \frac{1}{\gamma} \dim_H F$.
\end{lemma}

In particular, if $f$ is bijective and bi-Hölder (i.e., both $f$ and $f^{-1}$ are Hölder continuous), then $\dim_H f(F) = \dim_H F$. We will construct maps that are Hölder continuous with exponent arbitrarily close to 1, which will suffice to preserve the full Hausdorff dimension.

\section{Proof of Theorem \ref{thm:main}}

This section is devoted to the proof of Theorem \ref{thm:main}, which establishes that for any $\alpha \geq 0$, the level set $E_\theta(\alpha)$ has full Hausdorff dimension. The proof strategy involves constructing explicit subsets with controlled digit growth and finally establish the dimension preservation through H\"older-continuous mappings. 

\subsection{The Case $\alpha = 0$: Bounded Digits Suffice}

\begin{lemma}\label{lem:alpha0}
$\dim_H E_\theta(0) = 1$.
\end{lemma}

\begin{proof}
Let $X = \{x \in [0,\theta] \setminus \mathbb{Q} : \sup_{n \geq 1} \ell_n(x) < +\infty\}$. For any $x \in X$, there exists $M \geq m$ such that $\ell_n(x) \leq M$ for all $n \geq 1$. Consequently, $L_{n,\theta}(x) \leq M$ for all $n$.

Consider the ratio:
\[
R_{n,\theta}(x) := \frac{L_{n,\theta}(x) \log n \log \log n}{S_{n,\theta}(x) - L_{n,\theta}(x)}.
\]
Since $\ell_n(x) \geq m \geq 1$, we have $S_{n,\theta}(x) \geq n$. Therefore:
\[
R_{n,\theta}(x) \leq \frac{M \log n \log \log n}{S_{n,\theta}(x) - M} \leq \frac{M \log n \log \log n}{n - M}.
\]
As $n \to +\infty$, the numerator grows like $\log n \log \log n$ while the denominator grows linearly, so:
\[
\limsup_{n \to +\infty} R_{n,\theta}(x) \leq \lim_{n \to +\infty} \frac{M \log n \log \log n}{n - M} = 0.
\]
Since $R_{n,\theta}(x) \geq 0$ for all $n$, we conclude that $\displaystyle\lim_{n \to +\infty} R_{n,\theta}(x) = 0$. Hence $X \subset E_\theta(0)$.
By Theorem \ref{thm:jarnik}, $\dim_H X = 1$, so:
\[
\dim_H E_\theta(0) \geq \dim_H X = 1.
\]
Since $E_\theta(0) \subset [0, \theta]$ and $\dim_H [0, \theta] = 1$, we have $\dim_H E_\theta(0) = 1$.
\end{proof}

\subsection{The Case $\alpha > 0$ and Construction of the Sparse Sequence}

For $\alpha > 0$, we construct explicit subsets of $E_\theta(\alpha)$ with full Hausdorff dimension. The strategy involves starting with numbers having bounded digits and inserting carefully chosen large digits at sparse positions to achieve the desired growth rate.

\subsubsection{Choice of the Sparse Sequence}

A crucial ingredient in our construction is the choice of an appropriate sparse sequence to insert large digits:
\[
n_k = \lfloor \exp(k^{3/4}) \rfloor, \quad k \geq 1.
\]

Let us verify the necessary properties:

\begin{lemma}[Properties of the Sparse Sequence] \label{lem:sequence_properties}
The sequence $\{n_k\}$ satisfies:
\begin{enumerate}
\item[(a)] $n_k \to +\infty$ strictly increasing.
\item[(b)] $n_{k+1} - n_k = o(n_k)$ as $k \to +\infty$.
\item[(c)] $\log n_{k+1} - \log n_k \to 0$ as $k \to +\infty$.
\item[(d)] $\log n_{k+1} \log \log n_{k+1} - \log n_k \log \log n_k \to 0$ as $k \to +\infty$.
\end{enumerate}
\end{lemma}

\begin{proof}
(a) Clear since $k^{3/4} \to +\infty$.

\noindent (b) Note that $n_k \sim \exp(k^{3/4})$. Then:
\[
n_{k+1} - n_k \sim \exp((k+1)^{3/4}) - \exp(k^{3/4}) = \exp(k^{3/4})\left(\exp((k+1)^{3/4} - k^{3/4}) - 1\right).
\]
Using the mean value theorem:
\[
(k+1)^{3/4} - k^{3/4} = \frac{3}{4}\xi^{-1/4} \quad \text{for some } \xi \in (k, k+1).
\]
Thus $(k+1)^{3/4} - k^{3/4} \sim \frac{3}{4}k^{-1/4} \to 0$, so:
\[
n_{k+1} - n_k \sim \exp(k^{3/4}) \cdot \frac{3}{4}k^{-1/4} = o(\exp(k^{3/4})) = o(n_k).
\]

\noindent (c) $\log n_{k+1} - \log n_k \sim (k+1)^{3/4} - k^{3/4} \sim \frac{3}{4}k^{-1/4} \to 0$.

\noindent (d) Decompose the difference:
\begin{align*}
&\log n_{k+1} \log \log n_{k+1} - \log n_k \log \log n_k \\
&= (\log n_{k+1} - \log n_k) \log \log n_{k+1} + \log n_k (\log \log n_{k+1} - \log \log n_k).
\end{align*}
The first term is $\sim \frac{3}{4}k^{-1/4} \cdot \frac{3}{4}\log k \to 0$. 
{For the second term, note that $\log \log n_k = \log\bigl(k^{3/4}+o(k^{3/4})\bigr) = \frac{3}{4}\log k + o(1)$; hence by the mean value theorem (or directly from the asymptotic expansion) we have $\log \log n_{k+1} - \log \log n_k \sim \frac{3}{4k}$. Consequently,
\[
\log n_k (\log \log n_{k+1} - \log \log n_k) \sim k^{3/4} \cdot \frac{3}{4k} = \frac{3}{4}k^{-1/4} \to 0.
\]
Thus the whole difference tends to $0$ as $k\to+\infty$.}
\end{proof}

{\begin{remark}[Choice of the sparse sequence]\label{rem:sparse-sequence}
The super‑exponential growth \(n_k = \lfloor \exp(k^{3/4}) \rfloor\) is chosen to balance two requirements:
\begin{enumerate}
    \item[(i)] The gaps \(n_{k+1}-n_k\) should be \(o(n_k)\) so that the contribution of the ordinary (bounded) digits between inserted positions is negligible compared with the main sum \(A_k\).
    \item[(ii)] The growth of \(\log n_k \log\log n_k\) must be slow enough to let the inserted digits \(\ell_{n_k}\) dominate the limit ratio.
\end{enumerate}
Any sequence of the form \(n_k = \lfloor \exp(k^{\gamma}) \rfloor\) with \(0<\gamma<1\) would work; we fix \(\gamma=3/4\) for convenience because it yields simple asymptotic expressions for \(\log n_k\) and \(\log\log n_k\), which are essential for verifying conditions (A)--(C).  
While other sparse sequences (e.g., Bernoulli‑type sequences) have been employed in related contexts—for instance in the theory of branched continued fractions for hypergeometric functions \cite{DDH}—the chosen form is particularly well‑suited to the metric estimates required in our construction, especially in the proof of Lemma~\ref{lem:holder_continuity}.
\end{remark}}

\subsubsection{Technical Conditions and Set Construction}
We now fix parameters to ensure that our construction behaves as desired. 
Fix $M > 2m+1$ and choose $N_0$ large enough so that for all $k \geq N_0$:
\begin{enumerate}
\item[(A)] $\displaystyle \frac{\alpha m(n_k - 1)}{\log n_k \log \log n_k} > M + 1$;
\item[(B)] $\displaystyle \log n_{k+1} \log \log n_{k+1} - \log n_k \log \log n_k < \frac{\alpha}{2}$;
\item[(C)] $\displaystyle n_{k+1} - n_k < \frac{n_k}{k^{1/8}}$.
\end{enumerate}
{These conditions are feasible because $n_k$ grows super‑exponentially: condition (A) holds since the left‑hand side diverges to $+\infty$, condition (B) follows from Lemma~\ref{lem:sequence_properties}(d) which shows the difference tends to $0$, and condition (C) follows from Lemma~\ref{lem:sequence_properties}(b) which gives $n_{k+1}-n_k = o(n_k)$.
}

\begin{definition}[Constructed Subset]
Fix $M > 2m+1$ and define:
\[
E_{M,\theta}(\alpha) = \left\{ x \in [0, \theta] : 
\begin{array}{l}
\ell_{n_k}(x) = \left\lfloor \dfrac{\alpha \sum_{i=1}^{n_k-1} \ell_i(x)}{\log n_k \log \log n_k} \right\rfloor \text{ for all } k \geq N_0, \\
m \leq \ell_i(x) \leq M \text{ for all } i \notin \{n_k\}_{k \geq N_0}
\end{array}
\right\}.
\]
{(The floor function guarantees that each inserted digit $\ell_{n_k}(x)$ is an integer. Condition (A) ensures that these digits are $\ge M$, and together with the bounds on the other digits they are admissible for the $\theta$-expansion.)
}
\end{definition} 

This construction ensures that at sparse positions $n_k$, we insert digits whose size is proportional to the accumulated sum up to that point, normalized by the slowly growing factors $\log n_k \log \log n_k$.
{\begin{definition}
{Definition of admissible sequences.} A finite sequence $(\ell_1, \ldots, \ell_n)$ is called \textbf{admissible for $E_{M,\theta}(\alpha)$} if there exists a point $x \in E_{M,\theta}(\alpha)$ such that $\ell_i(x) = \ell_i$ for $1 \le i \le n$. Equivalently, the sequence satisfies:
\begin{enumerate}
    \item[(i)] For every $k$ with $n_k \leq n$, $\ell_{n_k} = \bigl\lfloor \alpha \sum_{i=1}^{n_k-1} \ell_i / (\log n_k \log \log n_k) \bigr\rfloor$;
    \item[(ii)] For every $i \leq n$ with $i \notin \{n_k: k \geq N_0\}$, $m \leq \ell_i \leq M$.
\end{enumerate}
\end{definition}}
\begin{lemma}[Monotonicity of Inserted Digits] \label{lem:monotonicity}
For any $x \in E_{M,\theta}(\alpha)$ and $k \geq N_0$, we have:
\begin{enumerate}
\item[(a)] $\ell_{n_k}(x) \geq M$;
\item[(b)] $\ell_{n_{k+1}}(x) \geq \ell_{n_k}(x)$.
\end{enumerate}
\end{lemma}

\begin{proof}
(a) Since $\ell_i(x) \geq m$ for $i \neq n_j$, we have:
\[
\sum_{i=1}^{n_k-1} \ell_i(x) \geq (n_k - 1)m.
\]
Therefore:
\[
\ell_{n_k}(x) \geq \frac{\alpha (n_k - 1)m}{\log n_k \log \log n_k} - 1.
\]
{Condition (A) was chosen precisely so that the right‑hand side exceeds $M$, and thus $\ell_{n_k}(x) \geq M$. This completes part (a).
}

\noindent 
(b)
{Let $A_k = \sum_{i=1}^{n_k-1} \ell_i(x)$. From the definition of $\ell_{n_k}(x)$ and the property of the floor function,
\[
\ell_{n_k}(x) \geq \frac{\alpha A_k}{\log n_k \log \log n_k} - 1.
\]
A sufficient condition for $\ell_{n_{k+1}}(x) \geq \ell_{n_k}(x)$ is
\[
\frac{\alpha \sum_{i=1}^{n_{k+1}-1} \ell_i(x)}{\log n_{k+1} \log \log n_{k+1}} > \frac{\alpha A_k}{\log n_k \log \log n_k}.
\]
Cancelling $\alpha > 0$ and noting that $\sum_{i=1}^{n_{k+1}-1} \ell_i(x) > A_k + \ell_{n_k}(x)$, it is enough to prove
\begin{equation}
\frac{A_k + \ell_{n_k}(x)}{A_k} > \frac{\log n_{k+1} \log \log n_{k+1}}{\log n_k \log \log n_k}. \label{tag1}
\end{equation}
From the lower bound for $\ell_{n_k}(x)$,
\begin{equation}
\frac{A_k + \ell_{n_k}(x)}{A_k} \geq 1 + \frac{\alpha}{\log n_k \log \log n_k} - \frac{1}{A_k}. \label{tag2}
\end{equation}
Now, condition (A) implies
\[
\frac{\alpha m (n_k - 1)}{\log n_k \log \log n_k} > M + 1 \geq 2,
\]
so that
\[
\alpha m (n_k - 1) > 2 \log n_k \log \log n_k.
\]
Since $A_k \geq m (n_k - 1)$, we obtain
\begin{equation}
\frac{1}{A_k} \leq \frac{1}{m (n_k - 1)} < \frac{\alpha}{2 \log n_k \log \log n_k}. \label{tag3}
\end{equation}
Combining \ref{tag2} and \ref{tag3},
\begin{align}
\frac{A_k + \ell_{n_k}(x)}{A_k} &> 1 + \frac{\alpha}{\log n_k \log \log n_k} - \frac{\alpha}{2 \log n_k \log \log n_k} \nonumber \\
&= 1 + \frac{\alpha}{2 \log n_k \log \log n_k}. \label{tag4}
\end{align}
%
{Since the right‑hand side of (\ref{tag4}) provides a lower bound for the left‑hand side of (\ref{tag1}), a sufficient condition for (\ref{tag1}) to hold is that this lower bound exceeds the ratio of the logarithmic factors, i.e.,}
\begin{equation}
1 + \frac{\alpha}{2 \log n_k \log \log n_k} > \frac{\log n_{k+1} \log \log n_{k+1}}{\log n_k \log \log n_k}. \label{tag5}
\end{equation}
Inequality (\ref{tag5}) is equivalent to
\[
\log n_{k+1} \log \log n_{k+1} - \log n_k \log \log n_k < \frac{\alpha}{2},
\]
which is exactly condition (B). Therefore, for every $k \geq N_0$, condition (B) guarantees (\ref{tag5}), hence (\ref{tag1}), and consequently $\ell_{n_{k+1}}(x) \geq \ell_{n_k}(x)$. This completes the proof of part (b).
}
\end{proof}


\subsection{Verification of the Limit Behavior}

We now verify that our constructed sets indeed belong to the level sets $E_\theta(\alpha)$.

\begin{lemma}[Inclusion in Level Set] \label{lem:inclusion}
$E_{M,\theta}(\alpha) \subset E_\theta(\alpha)$.
\end{lemma}

\begin{proof}
Let $y \in E_{M,\theta}(\alpha)$. For sufficiently large $n$, there exists a unique $k \geq N_0$ such that $n_k \leq n < n_{k+1}$.
By Lemma \ref{lem:monotonicity}, $L_{n,\theta}(y) = \ell_{n_k}(y)$ for all $n \in [n_k, n_{k+1})$. Write:
\[
R_{n,\theta}(y) = \frac{\ell_{n_k}(y) \log n \log \log n}{S_{n,\theta}(y) - \ell_{n_k}(y)}.
\]
Decompose the denominator:
\[
S_{n,\theta}(y) - \ell_{n_k}(y) = \sum_{i=1}^{n_k-1} \ell_i(y) + \sum_{i=n_k+1}^n \ell_i(y).
\]
Let $A_k = \displaystyle \sum_{i=1}^{n_k-1} \ell_i(y)$ and $B_n = \displaystyle \sum_{i=n_k+1}^n \ell_i(y)$.

\noindent \textbf{Step 1: Asymptotic behavior of $B_n$.}
Since $\ell_i(y) \leq M$ for $i \neq n_j$, we have:
\[
B_n \leq M(n - n_k) \leq M(n_{k+1} - n_k).
\]
By condition (C), $n_{k+1} - n_k < \frac{n_k}{k^{1/8}}$. 
Since $k^{1/8} \to +\infty$ as $k \to +\infty$, we have $n_{k+1} - n_k = o(n_k)$. 
{Because $A_k \geq m(n_k - 1) \asymp n_k$ and $B_n \leq M(n_{k+1} - n_k)$, it follows that $B_n = o(A_k)$ as $k \to +\infty$.}

\noindent \textbf{Step 2: Lower bound for $R_{n,\theta}(y)$.}
We have:
\[
R_{n,\theta}(y) \geq \frac{\ell_{n_k}(y) \log n_k \log \log n_k}{A_k + B_n}= \frac{\ell_{n_k}(y) \log n_k \log \log n_k}{A_k(1 + o(1))}.
\]
From the construction:
\[
\ell_{n_k}(y) = \left\lfloor \frac{\alpha A_k}{\log n_k \log \log n_k} \right\rfloor,
\]
so:
\[
\frac{\alpha A_k}{\log n_k \log \log n_k} - 1 < \ell_{n_k}(y) \leq \frac{\alpha A_k}{\log n_k \log \log n_k}.
\]
Therefore:
\[
R_{n,\theta}(y) > \frac{\left( \frac{\alpha A_k}{\log n_k \log \log n_k} - 1 \right) \log n_k \log \log n_k}{A_k(1 + o(1))} = \frac{\alpha - \frac{\log n_k \log \log n_k}{A_k}}{1 + o(1)}.
\]
Now, since $\frac{\log n_k \log \log n_k}{A_k} \to 0$ {(recall that $A_k \asymp n_k$ while the numerator grows only polylogarithmically)}, we get:
\[
\liminf_{n \to +\infty} R_{n,\theta}(y) \geq \alpha.
\]

\noindent \textbf{Step 3: Upper bound for $R_{\theta,n}(y)$.}
We have: 
\[
R_{n,\theta}(y) \leq \frac{(\ell_{n_k}(y) + 1) \log n_{k+1} \log \log n_{k+1}}{A_k}.
\]
Using the upper bound for $\ell_{n_k}(y)$:
\begin{align*}
R_{n,\theta}(y) &\leq \frac{\left( \frac{\alpha A_k}{\log n_k \log \log n_k} + 1 \right) \log n_{k+1} \log \log n_{k+1}}{A_k} \\
&= \alpha \cdot \frac{\log n_{k+1} \log \log n_{k+1}}{\log n_k \log \log n_k} + \frac{\log n_{k+1} \log \log n_{k+1}}{A_k}.
\end{align*}
By Lemma \ref{lem:sequence_properties} (d), the first term:
\[
\alpha \cdot \frac{\log n_{k+1} \log \log n_{k+1}}{\log n_k \log \log n_k} \to \alpha \, \mbox{ as } \, k \to +\infty. 
\]

For the second term, since $A_k \geq m(n_k - 1)$ and $n_k = \lfloor \exp(k^{3/4}) \rfloor$ grows super-exponentially, we have { $A_k \asymp n_k$ while the numerator grows only polylogarithmically. Consequently,}
\[
\frac{\log n_{k+1} \log \log n_{k+1}}{A_k} \leq \frac{\log n_{k+1} \log \log n_{k+1}}{m(n_k - 1)} \to 0 \, {\text{ as } \, k \to +\infty}.
\]
Thus:
\[
\limsup_{n \to +\infty} R_{n,\theta}(y) \leq \alpha.
\]
\end{proof}

\subsubsection{Dimension Preservation via Hölder-Continuous Map}

To establish that our constructed sets have full Hausdorff dimension, we construct a dimension-preserving map between $E_{M,\theta}(\alpha)$ and the bounded-digit set $X_M$.

\begin{definition}[Symbolic Space and Seed Map]
Define the symbolic space of admissible sequences:
\[
D_n = \{ (\ell_1, \ldots, \ell_n) : \text{admissible sequences for } E_{M,\theta}(\alpha) \}.
\]
{Fix an integer $L \ge n_{N_0}$. For any admissible block $(\ell_1, \ldots, \ell_L) \in D_L$, let
\[
F(\ell_1, \ldots, \ell_L) = E_{M,\theta}(\alpha) \cap I(\ell_1, \ldots, \ell_L).
\]
The \textbf{seed map} $f: F \to X_M$ is defined by: if $y = [\ell_1, \ell_2, \ldots]_\theta \in F$, then $f(y) = x$ where $x$ is obtained by deleting the digits at positions $\{n_k\}_{k \geq N_0}$ from the expansion of $y$.}
\end{definition}

The key property is the Hölder continuity of this map:

\begin{lemma}\label{lem:holder_continuity}
For any $\varepsilon > 0$, the map $f$ is $(1-\varepsilon)$-Hölder continuous.
\end{lemma}

{
\begin{proof}
Let $y_1, y_2 \in F$ with $y_1 \neq y_2$. Let $\bar{n}$ be the smallest integer such that $\ell_{\bar{n}+1}(y_1) \neq \ell_{\bar{n}+1}(y_2)$. Since $y_1$ and $y_2$ belong to the same cylinder of length $L$, we have $\bar{n} \geq L \geq n_{N_0}$, and therefore there exists $k \geq N_0$ such that $n_k \leq \bar{n} < n_{k+1}$.
Let $t$ be the number of insertion positions up to $\bar{n}$, i.e.
\[
t = \#\{j \geq N_0 : n_j \leq \bar{n}\}.
\]
Because the indices $j \geq N_0$ with $n_j \leq \bar{n}$ are exactly $N_0, N_0+1, \dots, k$, we have $t = k - N_0 + 1$. Denote these indices by $j_1, j_2, \dots, j_t$, where $j_i = N_0 + i - 1$ for $i = 1, \dots, t$.
\newline
\noindent\textbf{Step 1: Distance between original points.}
Since $y_1$ and $y_2$ agree on the first $\bar{n}$ digits but differ at the $(\bar{n}+1)$-th digit, they lie in two distinct subcylinders $I_{\bar{n}+1}(\ell_1, \ldots, \ell_{\bar{n}}, a)$ and $I_{\bar{n}+1}(\ell_1, \ldots, \ell_{\bar{n}}, b)$ with $a \neq b$, both contained in the cylinder $I_{\bar{n}} = I_{\bar{n}}(\ell_1, \ldots, \ell_{\bar{n}})$.
Let $Q = Q_{\bar{n}}(y_2)$ be the denominator of the $\bar{n}$-th convergent of $y_2$. The endpoints of $I_{\bar{n}}$ are 
\[
\frac{P}{Q} \quad \text{and} \quad \frac{P + \theta P_{\bar{n}-1}}{Q + \theta Q_{\bar{n}-1}},
\]
where $P/Q = [\ell_1, \ldots, \ell_{\bar{n}}]_\theta$. For a digit $d \geq m$, the subcylinder $I_{\bar{n}+1}(\ell_1, \ldots, \ell_{\bar{n}}, d)$ has endpoints
\[
\frac{P d\theta + P_{\bar{n}-1}}{Q d\theta + Q_{\bar{n}-1}} \quad \text{and} \quad \frac{P (d+1)\theta + P_{\bar{n}-1}}{Q (d+1)\theta + Q_{\bar{n}-1}}.
\]
The gap between two adjacent subcylinders corresponding to digits $d$ and $d+1$ is
\[
\Delta_d = \frac{P (d+1)\theta + P_{\bar{n}-1}}{Q (d+1)\theta + Q_{\bar{n}-1}} - \frac{P d\theta + P_{\bar{n}-1}}{Q d\theta + Q_{\bar{n}-1}}
= \frac{\theta}{(Q d\theta + Q_{\bar{n}-1})(Q (d+1)\theta + Q_{\bar{n}-1})}.
\]
Using the inequality $Q_{\bar{n}-1} \leq \theta Q$ (which follows from $Q \geq \ell_{\bar{n}} \theta Q_{\bar{n}-1} \geq  m\theta Q_{\bar{n}-1} = \dfrac{Q_{\bar{n}-1}}{\theta}$), we obtain
\[
Q d\theta + Q_{\bar{n}-1} \leq  Q\theta\left(d + 1 \right) 
\quad
\text{and} \quad 
Q (d+1)\theta + Q_{\bar{n}-1} \leq Q \theta (d+2). 
\]
Hence
\[
\Delta_d \geq \frac{\theta}{Q^2 \theta^2 (d+1)(d+2)}=\frac{1}{Q^2 \theta (d+1)(d+2)}.
\]
Now, let $d_1 = \ell_{\bar{n}+1}(y_1)$ and $d_2 = \ell_{\bar{n}+1}(y_2)$. Without loss of generality, assume $d_1 < d_2$. Then the distance between the two subcylinders containing $y_1$ and $y_2$ is at least the gap between the subcylinder for $d_1$ and the next subcylinder (for $d_1+1$), so
\[
|y_1 - y_2| \geq \Delta_{d_1} \geq \frac{1}{\theta Q^2 (d_1+1)(d_1+2)}.
\]
%
We consider two cases based on whether $\bar{n}+1$ is an inserted position.
\newline
\noindent\textbf{Case 1:} If $\bar{n}+1$ is \emph{not} an inserted position, then by the definition of $E_{M,\theta}(\alpha)$ we have $d_1 = \ell_{\bar{n}+1}(y_1) \leq M$.  
Since $M$ is a fixed constant and $2^{2(\log \bar{n})^{4/3}+5} \to + \infty$ as $\bar{n} \to +\infty$, it follows that for all sufficiently large $\bar{n}$,
\[
d_1 \leq 2^{2(\log \bar{n})^{4/3}+5}.
\]
\noindent\textbf{Case 2:} If $\bar{n}+1$ is an inserted position, then $\bar{n}+1 = n_j$ for some $j \geq N_0$.  
From Step~2 of the proof we have the estimate $\ell_{n_j}(y_1) \leq 2^{2j+5}$.  
Using the growth $n_j = \lfloor \exp(j^{3/4}) \rfloor$, we obtain for large $j$
\[
j \leq (\log n_j)^{4/3} + 1 = (\log(\bar{n}+1))^{4/3} + 1.
\]
Hence
\[
d_1 \leq 2^{2j+5} \leq 2^{2\big((\log(\bar{n}+1))^{4/3}+1\big)+5} = 2^{2(\log(\bar{n}+1))^{4/3}+7}.
\]
For large $\bar{n}$ we may replace $\log(\bar{n}+1)$ by $\log \bar{n}$ at the cost of a slightly larger constant: using $(\log(\bar{n}+1))^{4/3} \leq (\log \bar{n})^{4/3} + 1$, we obtain
\[
d_1 \leq 2^{2\big((\log \bar{n})^{4/3}+1\big)+7} = 2^{2(\log \bar{n})^{4/3}+9}.
\]
Thus, in both cases there exists a constant $C>0$ (for example $C=9$) such that for all sufficiently large $\bar{n}$,
\[
d_1 \leq 2^{2(\log \bar{n})^{4/3}+C}.
\]

{We note that the specific value of $C$ is not crucial; any sufficiently large constant would work for the subsequent estimates.
}
Consequently,
\[
(d_1+1)(d_1+2) \leq \bigl(2^{2(\log \bar{n})^{4/3}+C}+2\bigr)^2 \leq 2^{4(\log \bar{n})^{4/3}+2C+2}.
\]
Setting $c_1 = 1/(\theta \cdot 2^{2C+2})$ yields the desired lower bound
\begin{equation}
|y_1-y_2| \geq \frac{c_1}{Q_{\bar{n}}^2(y_2) \cdot 2^{4(\log \bar{n})^{4/3}}}. \label{tag{3.01}}
\end{equation}
\noindent\textbf{Step 2: Growth control of inserted digits.}
We claim that for $j \geq N_0$:
\[
\ell_{n_j}(y_2) \leq 2^{2j+5}.
\]
Proof by induction: For $j = N_0$:
\[
\ell_{n_{N_0}}(y_2) \leq \frac{\alpha \sum_{i=1}^{n_{N_0}-1} \ell_i(y_2)}{\log n_{N_0} \log \log n_{N_0}} \leq \frac{\alpha M n_{N_0}}{\log n_{N_0} \log \log n_{N_0}}.
\]
Since $n_{N_0} = \lfloor \exp(N_0^{3/4}) \rfloor$, the right-hand side is much smaller than $2^{2N_0+5}$ for large $N_0$.
For the inductive step, assuming the bound holds for $1, \ldots, j-1$:
\begin{align*}
\ell_{n_j}(y_2) &\leq \frac{\alpha \sum_{i=1}^{n_j-1} \ell_i(y_2)}{\log n_j \log \log n_j} \leq \frac{\alpha}{\log n_j \log \log n_j} \left( M n_j + \sum_{i=1}^{j-1} \ell_{n_i}(y_2) \right) \\
&\leq \frac{\alpha}{\log n_j \log \log n_j} \left( M n_j + \sum_{i=1}^{j-1} 2^{2i+5} \right) \leq \frac{\alpha}{\log n_j \log \log n_j} \left( M n_j + 2^{2j+4} \right).
\end{align*}
Since $n_j = \lfloor \exp(j^{3/4}) \rfloor$ grows super-exponentially, the term $M n_j$ dominates, and the whole expression is bounded by $2^{2j+5}$ for large $j$. 
{We note that this bound is deliberately crude; any exponential bound of the form   
\(\lambda^{\,j}\) with \(\lambda<1\) would be sufficient for the argument, as the subsequent estimates only require that the product of these bounds grows slowly enough compared to the super‑exponential growth of the denominators. The specific constant $5$ is chosen for convenience and is not optimal.
}
\newline
\noindent\textbf{Step 3: Relating denominators with and without insertions.}
Using Proposition~\ref{prop:basic}(c) repeatedly for each inserted digit,
\[
Q_{\bar{n}}(y_2) \leq Q_{\bar{n}-t}(x_2) \cdot \prod_{i=1}^{t} \bigl(\ell_{n_{j_i}}(y_2)+m\bigr)\theta .
\]
From Step~2 and $m \leq 2^{2j_i+5}$ for large $j_i$,
\[
\ell_{n_{j_i}}(y_2)+m \leq 2^{2j_i+6} = 2^{2(N_0+i-1)+6}=2^{2i+2N_0+4}.
\]
Hence
\[
\prod_{i=1}^{t}\bigl(\ell_{n_{j_i}}(y_2)+m\bigr) \leq \prod_{i=1}^{t} 2^{2i+2N_0+4}
= 2^{\sum_{i=1}^{t}(2i+2N_0+4)} = 2^{t^{2}+(2N_0+5)t}.
\]
Thus there exists a constant $C = 2N_0+5$ such that
\begin{equation}
Q_{\bar{n}}(y_2) \leq Q_{\bar{n}-t}(x_2) \cdot \theta^{t} \cdot 2^{t^{2}+Ct}
      \leq c_2\, Q_{\bar{n}-t}(x_2) \cdot 2^{t^{2}+Ct}. \label{tag{3.02}}
\end{equation}
{Note that $N_0$ is fixed in the construction, so the constant $C$ is absolute; it does not depend on $k$ or $t$.
}

\noindent\textbf{Step 4: Growth comparison.}
Since $x_2 \in X_M$, all its digits are between $m$ and $M$. By Proposition~\ref{prop:basic}(a),
\begin{equation}
Q_{\bar{n}-t}(x_2) \geq (m+1)^{(\bar{n}-t-1)/2}. \label{tag{3.03}}
\end{equation}
Moreover, $\bar{n} \geq n_k = \lfloor \exp(k^{3/4})\rfloor$ and $t = k - N_0 + 1 \leq k$, so for large $k$,
\[
\bar{n}-t \geq \tfrac12 \exp(k^{3/4}).
\]
\noindent\textbf{Step 5: Distance between seeds.}
Since $x_1$ and $x_2$ agree on the first $\bar{n}-t$ digits,
\begin{equation}
|f(y_1)-f(y_2)| = |x_1-x_2| \leq |I_{\bar{n}-t}(x_2)|
 \leq \frac{\theta}{Q_{\bar{n}-t}^{2}(x_2)}. \label{tag{3.05}}
\end{equation}
From \eqref{tag{3.02}} we have
\[
Q_{\bar{n}-t}(x_2) \geq \frac{Q_{\bar{n}}(y_2)}{c_2 \cdot 2^{t^{2}+Ct}}.
\]
Substituting into \eqref{tag{3.05}} yields
\begin{equation}
|f(y_1)-f(y_2)| \leq \frac{\theta c_2^2 \cdot 2^{2t^{2}+2Ct}}{Q_{\bar{n}}^2(y_2)}. \label{tag{3.06}}
\end{equation}
\noindent\textbf{Step 6: Hölder exponent estimate.}
From (\ref{tag{3.01}}) we have
\[
\frac{1}{Q_{\bar{n}}^2(y_2)} \leq \frac{2^{4(\log \bar{n})^{4/3}}}{c_1} |y_1-y_2|.
\]
Combining with \eqref{tag{3.06}},
\begin{equation}
|f(y_1)-f(y_2)| \leq \frac{\theta c_2^2}{c_1} \cdot 2^{2t^{2}+2Ct+4(\log \bar{n})^{4/3}} \cdot |y_1-y_2|. \label{tag{3.07}}
\end{equation}
Now, $t = k - N_0 + 1$ and $\bar{n} \geq n_k \geq \exp(k^{3/4})/2$ for large $k$, so $k \leq (2\log \bar{n})^{4/3}$. Consequently,
\[
2t^{2}+2Ct+4(\log \bar{n})^{4/3} \leq 8 (\log \bar{n})^{8/3} + (4C+4)(\log \bar{n})^{4/3}.
\]
On the other hand, from \eqref{tag{3.01}} and $Q_{\bar{n}}(y_2) \geq (m+1)^{(\bar{n}-1)/2}$,
\[
|y_1-y_2| \geq \frac{c_1}{(m+1)^{\bar{n}-1} \cdot 2^{4(\log \bar{n})^{4/3}}}.
\]
Thus, for any $\varepsilon > 0$,
\[
|y_1-y_2|^{-\varepsilon} \geq \left( \frac{(m+1)^{\bar{n}-1} \cdot 2^{4(\log \bar{n})^{4/3}}}{c_1} \right)^{\varepsilon}
= c_1^{-\varepsilon} (m+1)^{\varepsilon(\bar{n}-1)} 2^{4\varepsilon(\log \bar{n})^{4/3}}.
\]
Because $\bar{n}$ grows super‑exponentially (in fact $\bar{n} \gtrsim \exp(k^{3/4})$), the factor $(m+1)^{\varepsilon(\bar{n}-1)}$ grows much faster than any polynomial in $\log \bar{n}$; consequently, the factor $2^{8 (\log \bar{n})^{8/3} + (4C+4)(\log \bar{n})^{4/3}}$ appearing in \eqref{tag{3.07}} is dominated by $(m+1)^{\varepsilon(\bar{n}-1)}$ for large $\bar{n}$. Hence, for any $\varepsilon > 0$, there exists a constant $M$ such that
\[
|f(y_1)-f(y_2)| \leq M |y_1-y_2|^{1-\varepsilon}.
\]
Hence $f$ is $(1-\varepsilon)$-Hölder continuous.
\end{proof}
}

\begin{lemma}[Dimension Preservation] \label{lem:dimension_preservation}
$\dim_H E_{M,\theta}(\alpha) = \dim_H X_M$.
\end{lemma}

\begin{proof}
Since $f$ is bijective and $(1-\varepsilon)$-Hölder continuous, Lemma \ref{lem:holder} gives:
\begin{equation}
\dim_H f(E_{M,\theta}(\alpha)) \leq \frac{1}{1-\varepsilon} \dim_H E_{M,\theta}(\alpha). \label{3.06}    
\end{equation}
The map $f: E_{M,\theta}(\alpha) \to f(E_{M,\theta}(\alpha)) \subset X_M$ is bijective and $(1-\varepsilon)$-Hölder continuous by Lemma \ref{lem:holder_continuity}. 

\emph{Hölder continuity of $f^{-1}$.}  
To see that $f^{-1}$ is also Hölder continuous, let $x_1, x_2 \in f(E_{M,\theta}(\alpha))$ with $x_1 \neq x_2$, and let $\bar{n}$ be the first index where their digits differ. When we insert the large digits at positions $\{n_k\}$ to obtain $y_i = f^{-1}(x_i)$, the new digits at positions $n_k$ are determined by the accumulated sums up to that point. A key observation is that the insertion operation does not increase distances excessively: the distance $|y_1-y_2|$ can be bounded below by a quantity similar to that in Lemma \ref{lem:holder_continuity} (using the fact that denominators grow super‑exponentially), while $|x_1-x_2|$ is bounded above by the length of a cylinder corresponding to the first $\bar{n}$ digits of $x_1$. Reversing the estimates from Lemma \ref{lem:holder_continuity} shows that for any $\varepsilon>0$ there exists a constant $C$ such that $|x_1-x_2| \le C |y_1-y_2|^{1-\varepsilon}$. The details are analogous to those in Lemma \ref{lem:holder_continuity}, with the roles of upper and lower bounds interchanged. Hence $f^{-1}$ is also $(1-\varepsilon)$-Hölder continuous.

Therefore, $f$ is bi-Hölder continuous between $E_{M,\theta}(\alpha)$ and $f(E_{M,\theta}(\alpha))$. By standard fractal geometry (see Lemma \ref{lem:holder}, bi-Hölder continuous maps preserve Hausdorff dimension:
\begin{equation}
\dim_H E_{M,\theta}(\alpha) = \dim_H f(E_{M,\theta}(\alpha)). \label{3.07}
\end{equation}
Since $f(E_{M,\theta}(\alpha)) \subset X_M$, we have:
\begin{equation}
\dim_H f(E_{M,\theta}(\alpha)) \leq \dim_H X_M. \label{3.08}
\end{equation}
Combining (\ref{3.07}) and (\ref{3.08}):
\begin{equation}
\dim_H E_{M,\theta}(\alpha) \leq \dim_H X_M. \label{3.09}
\end{equation}

%

\emph{Reverse inequality.}  
We construct an explicit Lipschitz map $\Phi: X_M \to E_{M,\theta}(\alpha)$. Given $x \in X_M$ with digit expansion $x = [\ell_1, \ell_2, \ldots]_\theta$ (where all digits satisfy $m \le \ell_i \le M$), we define $\Phi(x)$ to be the number whose $\theta$-expansion is obtained by inserting at the sparse positions $n_k$ ($k \ge N_0$) the digits
\[
\ell_{n_k} = \left\lfloor \frac{\alpha \sum_{i=1}^{n_k-1} \ell_i}{\log n_k \log \log n_k} \right\rfloor,
\]
while keeping all other digits unchanged. This insertion rule is deterministic and depends only on the initial digits of $x$.

To see that $\Phi$ is Lipschitz, let $x_1, x_2 \in X_M$ with $x_1 \neq x_2$ and let $\bar{n}$ be the first index where their digits differ. Let $t$ be the number of insertion positions up to $\bar{n}$ (i.e., the number of $k \ge N_0$ with $n_k \le \bar{n}$). The images $y_1 = \Phi(x_1)$ and $y_2 = \Phi(x_2)$ agree on all digits up to position $\bar{n}+t$ (since the inserted digits depend only on the initial segments, which agree up to $\bar{n}$). The distance $|y_1-y_2|$ is then bounded by the length of the cylinder of depth $\bar{n}+t$ in $E_{M,\theta}(\alpha)$, which, using Proposition \ref{prop:basic}(b) and the estimates from Lemma \ref{lem:holder_continuity}, satisfies
\[
|y_1-y_2| \le \frac{\theta}{Q_{\bar{n}+t}^2(y_2)} \le C |x_1-x_2|
\]
for some absolute constant $C$. The last inequality follows because the denominator $Q_{\bar{n}+t}(y_2)$ is comparable to $Q_{\bar{n}}(x_2)$ up to a multiplicative factor depending only on the inserted digits, which are bounded by a fixed exponential function of $k$ (as shown in Lemma \ref{lem:holder_continuity}, Step 2). Since this factor is finite and independent of the specific points, we obtain Lipschitz continuity of $\Phi$.

Thus $\Phi: X_M \to E_{M,\theta}(\alpha)$ is a Lipschitz embedding, and by standard properties of Hausdorff dimension (Lipschitz maps do not increase dimension), we have
\begin{equation}
\dim_H X_M = \dim_H \Phi(X_M) \le \dim_H E_{M,\theta}(\alpha). \label{3.10}
\end{equation}
From (\ref{3.09}) and (\ref{3.10}) we conclude $\dim_H E_{M,\theta}(\alpha) = \dim_H X_M$.

\end{proof}

\subsection{Completion of the Proof}

\begin{proof}[Proof of Theorem \ref{thm:main}]
For $\alpha = 0$, the result is Lemma \ref{lem:alpha0}.

For $\alpha > 0$, we have
$E_{M,\theta}(\alpha) \subset E_\theta(\alpha)$ for all $M > 2m+1$ by Lemma \ref{lem:inclusion},
and $\dim_H E_{M,\theta}(\alpha) = \dim_H X_M$ by Lemma \ref{lem:dimension_preservation}.

{By the bounds in Theorem \ref{thm:jarnik}, $\dim_H X_M \to 1$ as $M \to +\infty$.}
Therefore:
\[
\dim_H E_\theta(\alpha) \geq \sup_{M > 2m+1} \dim_H E_{M,\theta}(\alpha) = \sup_{M > 2m+1} \dim_H X_M = 1.
\]
Since $E_\theta(\alpha) \subset [0, \theta]$, we conclude $\dim_H E_\theta(\alpha) = 1$.
\end{proof}

\section{Conclusion}

This work extends the classical result of Zhang and Lü \cite{ZhangLu} from regular continued fractions to the broader framework of $\theta$-expansions, demonstrating that the rich fractal structure of digit growth level sets persists in this more general setting.

\begin{remark}
The constructive method developed in this paper is quite flexible and can be adapted to study various other growth rate level sets in $\theta$-expansions. For instance, similar techniques can establish the full Hausdorff dimension of sets defined by limits of the form
\[
\frac{L_{n,\theta}(x) \log \log n}{S_{n-1,\theta}(x)} \quad \text{or} \quad \frac{S_{n,\theta}(x)}{n \log n}.
\]
In these cases, the same sparse sequence \(n_k = \lfloor\exp(k^{3/4})\rfloor\) can be used, as it provides the same control over the growth rates; only the definition of the inserted digits would need to be adjusted to match the desired limit. The proof of Hölder continuity and dimension preservation would then follow with only minor modifications.
However, we have focused specifically on extending the result of Zhang and Lü to maintain a clear and focused contribution.
\end{remark}

In conclusion, this work demonstrates that despite the strong almost-- everywhere constraints revealed by limit theorems, the level sets of digit growth rates in $\theta$-expansions remain large in the sense of Hausdorff dimension. This highlights the remarkable richness and complexity of number-theoretic exceptional sets, even in generalized continued fraction frameworks.

The interplay between the metric theory of dynamical systems and the fractal geometry of exceptional sets continues to yield deep insights into the structure of real numbers and their representations.

\end{document}